\documentclass[aap,preprint]{imsart}

\RequirePackage[OT1]{fontenc}
\RequirePackage{amsthm,amsmath,natbib}
\RequirePackage[colorlinks,citecolor=blue,urlcolor=blue]{hyperref}
\RequirePackage{hypernat}
\usepackage{bm}
\arxiv{math.PR/0000000}

\startlocaldefs
\numberwithin{equation}{section}
\theoremstyle{plain}
\newtheorem{theorem}{Theorem}[section]
\newtheorem{proposition}{Proposition}
\newtheorem{lemma}{Lemma}

\endlocaldefs

\begin{document}

\begin{frontmatter}
\title{Formulas for the Laplace Transform of Stopping Times based on Drawdowns and Drawups}
\runtitle{Drawdowns and Drawups}

\begin{aug}
\author{\fnms{Hongzhong} \snm{Zhang}\thanksref{m1}\ead[label=e1]{hzhang3@gc.cuny.edu}}
\and
\author{\fnms{Olympia} \snm{Hadjiliadis}\thanksref{t1,m1,m2}\ead[label=e2]{ohadjiliadis@brooklyn.cuny.edu}}

\thankstext{t1}{Corresponding author}
\runauthor{H. Zhang and O.Hadjiliadis}

\affiliation{Graduate Center\thanksmark{m1} and Brooklyn College\thanksmark{m2}, C.U.N.Y.}

\address{Department of Mathematics\\365 Fifth Avenue\\
New York, NY 10016\\
\printead{e1}\\
\phantom{E-mail:\ }\printead*{e2}}

\end{aug}
\begin{abstract}
In this work we study drawdowns and drawups of general diffusion processes. The drawdown process is defined as the current drop of the process from its running maximum, while the drawup process is defined as the current increase over its running minimum.  The drawdown and the drawup are the first hitting times of the drawdown and the drawup processes respectively. In particular, we derive a closed-form formula for the Laplace transform of the probability density of the drawdown of $a$ units when it precedes the drawup of $b$ units. We then separately consider the special case of drifted Brownian motion, for which we derive a closed form formula for the above-mentioned density by inverting the Laplace transform. Finally, we apply the results to a problem of interest in financial risk-management and to the problem of transient signal detection and identification of two-sided changes in the drift of general diffusion processes.
\end{abstract}

\end{keyword}\begin{keyword}[class=AMS]
\kwd[Primary ]{60G40}
\kwd[; secondary ]{60J60}
\end{keyword}

\begin{keyword}
\kwd{Drawdowns and Drawups; Diffusion Process;  Stopping Time}
\end{keyword}

\end{frontmatter}

\section{Introduction}
\label{intro}
In this paper we derive the Laplace transform of the probability density of the drawdown of $a$ units when it precedes the drawup of $b$ units for a general diffusion process. The drawdown process is defined as the current drop of the process from its running maximum, while the drawup process is defined as the current increase over its running minimum.  The drawdown and the drawup are then the first hitting times of the drawdown and the drawup processes respectively. The derivation is first accomplished in the case that $a=b$, by drawing the connection of the relevant event to the range process. We then consider the case $a\neq b$ and derive the Laplace transform through path decomposition. A key element in the above derivation is a function related to the first hitting times of the underlying diffusion process. We then consider the special case of drifted Brownian motion with drift $\mu$ and volatility $\sigma$. In this case we are able to invert the Laplace transform and derive the density of the drawdown of $a$ units when it precedes the drawup of $b$ units, which we denote by $p^{(\mu)}(t;a,b)$. Finally, we discuss the applications of the results to a problem of interest in financial risk-management and to the problem of transient signal detection and identification of two-sided changes in the drift of general diffusion processes.

Our results extend the work of Taylor \cite{Taylor} and Lehoczky \cite{JPL77}.
Taylor \cite{Taylor} derives the joint Laplace transform of the drawdown and the maximum stopped at the drawdown in a drifted Brownian motion model. Lehoczky \cite{JPL77} extends the above result in the case of a diffusion process. However, none of these results relate the drawdown to the drawup. In a recent paper of Salminen \& Vallois \cite{SalmVall}, the joint distribution of the maximum drawdown and the maximum drawup processes is studied in a drifted Brownian motion  model; yet it is not possible to extract information on the joint distribution of the drawdown and the drawup from this paper. In Hadjiliadis \& Vecer \cite{HadjVece06}, a closed-form formula is derived for the probability that the drawdown precedes the drawup in a drifted Brownian motion model. This result is later extended to diffusion processes in Pospisil, Vecer \& Hadjiliadis \cite{Posp09}. In Zhang \& Hadjiliadis \cite{ZhanHadj09}, the authors obtain the probability that the drawup of $a$ units precedes the drawdown of equal units in a drifted Brownian motion model in a finite time-horizon. However, the approach used there only applies to a drifted Brownian motion model, cannot be extended to a general diffusion process, and does not work in the general case $a\neq b$. In this paper, we supersede these limitations through the Laplace transform method.

Drawdown processes have been extensively used in the financial risk-management literature. Grossman \& Zhou \cite{GrosZhou}, Cvitanic \& Karatzas \cite{CvitKara}, Chekhlov, Uryasev \& Zabarankin \cite{ChekUryaZaba} studied portfolio optimization under constraints on the drawdown process. Magdon-Ismail et.\ al. \cite{Magdetal} determined the distribution of the maximum drawdown process of Brownian motion, based on which they described another time-adjusted measure of performance known as the the Calmar ratio (see \cite{MagdAtiy}). Meilijson \cite{Meil} proved that the drawdown can be viewed as the optimal exercise time of a certain type of look-back American put option. Other works which describe drawdown processes as dynamic measures of risk include Vecer \cite{Vece06,Vece07}, Pospisil \& Vecer \cite{PospVece07}, Pospisil, Vecer \& Xu \cite{PospVeceXu}, Zhang \& Hadjiliadis \cite{ZhanHadj09}. An overview of the existing techniques for analysis of market crashes as well as a collection of empirical studies of the drawdown process and the maximum drawdown process please refer to Sornette \cite{Sorn}.

Drawdown processes do not only provide dynamic measures of risk, but can also be viewed as measures of ``relative regret''. Similarly drawup processes can be viewed as measures of ``relative satisfaction''. Thus the a drawdown or a drawup of a certain number of units may signal the time in which an investor may choose to change his/her investment position depending on his/her perception of future moves of the market and his/her risk aversion. Using the results in our paper we are able to calculate the probability that a relative drawdown of $(100\times\alpha)\%$ occurs before a relative drawup of $(100\times\beta)\%$ in a finite time-horizon. On the other hand, a digital option on the event that the relative drawdown occurs before the relative drawup could also be seen as a means of protection. Our paper provides a closed-form formula for the risk-neutral price of this digital option at time 0 both in the case of an infinite maturity and in the case of a finite maturity.

Drawdown and drawup processes have also been used in the problem of quickest detection of abrupt changes in a stochastic process. More specifically, consider the situation in which a diffusion process is sequentially observed. At some unknown point in time, possibly as a result of the onset of a signal, the dynamics of the process change abruptly in one of two possible opposite directions in the drift. Drawdowns and drawups then provide a detection mechanism of the change point for each of the possible changes. They are known as CUSUM stopping times in the statistics literature, and their properties have been extensively studied (see Barnard \cite{Barn59}, Dobben \cite{Dobben68}, Bissell \cite{Biss69}, Woodall \cite{Wood84}, Hadjiliadis \& Moustakides \cite{HadjMous06}, Khan \cite{Khan08}, and Poor \& Hadjiliadis \cite{PHadj08}).  A challenging problem in engineering is the detection and identification of such signals when they are only present for a finite period of time. These signals are known as transient signals. Using the results in this paper, it is possible to derive closed-form formulas for the probability of misidentification of the direction of the change in the drift when the signal has exponential life. Moreover, using our results for drifted Brownian motion, we derive this probability when the transient signal is present for a finite period of time $T$.

The paper is structured in the following way: definitions and a fundamental lemma are introduced in Section 2. In Section 3, we derive the Laplace transform of the drawdown of $a$ units when it precedes the drawup of $b$ units, in the cases $a=b$ (Theorem \ref{thm1}), $a>b$ (Theorem \ref{thm2}) and $a<b$ (Theorem \ref{thm3}). A special case of a drifted Brownian motion model is discussed in Section 4, where we also derive the closed-form density $p^{(\mu)}(t;a,b)$  by analytical inversion of the Laplace transform. We then present applications of our results in a problem of risk-management and the problem of transient signal detection and identification of two-sided alternatives in Section 5. Finally, we conclude with some closing remarks in Section 6.
\section{Drawdown and Drawup processes}
\label{define}
We begin with a mathematical definition of a drawdown and a drawup in a diffusion model and present the fundamental lemma.

Consider an interval $I=(l,r)$, where $-\infty\le l<r\le\infty$. Let $(\Omega, \mathcal{F}, P)$ be a probability space, $\{B_t; t\ge0\}$ a Brownian motion, and $\{X_t;t\ge 0\}$ the unique strong solution of the following stochastic differential equation:
\begin{eqnarray}
dX_t&=&\mu(X_t)dt+\sigma(X_t)dB_t,~X_0=x\in I,
\end{eqnarray}
where $X_t\in I$ for all $t\ge0$, and $\sigma(u)>0$ for all $u\in I$. We will use $P_x(\cdot)$ to denote $P(\cdot|X_0=x)$.

The drawdown and drawup processes are defined respectively as
\begin{eqnarray}
DD_t&=&\sup_{0\le s\le t}X_s-X_t,\\
DU_t&=&X_t-\inf_{0\le s\le t}X_s.
\end{eqnarray}
The drawdown of $a$ units and the drawup of $b$ units are then defined respectively as
\begin{eqnarray}
T_D(a)&=&\inf\{t\ge0|DD_t=a\}, ~a>0,\\
T_U(b)&=&\inf\{t\ge 0|DU_t=b\}, ~b>0,
\end{eqnarray}
where, by convention, we assume that $\inf\phi=\infty$.

In the following section, we derive the main results in this paper. We need the following fundamental lemma to finish the proofs.
\begin{lemma} Let us denote  by $\tau_u$, $u\in I$, the first hitting time of the process $\{X_t; t\ge0\}$ to $u$. That is,
\begin{eqnarray*}
\tau_u=\inf\{t> 0|X_t=u\},
\end{eqnarray*}
For $y\le x\le z$ and $\lambda\ge0$, define
\begin{eqnarray}
E_x\left[e^{-\lambda\tau_y}\cdot1_{\{\tau_y<\tau_z\}}\right]&:=&\ell^X(y,z;x,\lambda).
\end{eqnarray}
Then
\begin{eqnarray}\label{ell}
\ell^X(y,z;x,\lambda)=\frac{g(x;\lambda)h(z;\lambda)-g(z;\lambda)h(x;\lambda)}{g(y;\lambda)h(z;\lambda)-g(z;\lambda)h(y;\lambda)}
\end{eqnarray}
with $g(\cdot;\lambda)$ and $h(\cdot;\lambda)$ being any two independent solutions of the ordinary differential equation
\begin{eqnarray}
\frac{1}{2}\sigma^2(u)\frac{\partial^2 f}{\partial
u^2}+\mu(u)\frac{\partial f}{\partial u}=\lambda f.
\end{eqnarray}
\end{lemma}
\begin{proof}
See \cite{JPL77}, page 603.
\end{proof}

\section{Mathematical Results}
\label{result}
In this section we derive formulas for the Laplace transform of the probability density of the drawdown of $a$
units when it precedes the drawup of $b$ units, for any $a,b>0$ satisfying $x\pm a,x\pm b\in \bar{I}=[l,r]$. We have
\begin{eqnarray}\label{lap}
~~\qquad E_x\left[e^{-\lambda
T_D(a)}\cdot1_{\{T_D(a)<T_U(b)\}}\right]&=&\int_0^\infty e^{-\lambda
t}P_x(T_D(a)\in dt, T_U(b)>t).
\end{eqnarray}
In the sequel we denote the Laplace transform (\ref{lap}) by $L_x^X(\lambda; a,b)$.
\subsection{The case of $a=b>0$}
We determine $L_x^X(\lambda;a,a)$ for $a>0$ in this paragraph.
\begin{theorem}\label{thm1}
For $a>0$ and $\lambda>0$, we have
\begin{eqnarray}
\label{lequal}L_x^{X}(\lambda;a,a)=\int_{x-a}^x\frac{\partial}{\partial
a}\ell^X(u,u+a;x,\lambda)du.
\end{eqnarray}
\end{theorem}
\begin{proof}
First, it is easily seen that for $t>0$ and $a>0$,
\begin{eqnarray}
\{T_D(a)\in dt, T_U(a)>t\}&=&\{\rho(a)\in dt, x-a<X_t<x\},
\end{eqnarray}
where $\rho(a)$ is the first hitting time of the range process
\begin{eqnarray}
\rho(a)&=&\inf\{t>0|R_t=a\},
\end{eqnarray}
with
\begin{eqnarray}
R_t&:=&\sup_{s\le t}X_s-\inf_{s\le t}X_s~=~DU_t+DD_t.
\end{eqnarray}
So it suffices to determine
\begin{eqnarray*}
E_x\left[ e^{-\lambda t}\cdot 1_{\left\{T_D(a)\in dt,T_U(a)>t,
X_t=u\right\}}\right],
\end{eqnarray*}
for all $x-a<u<x$. For this purpose observe that for any $t>0, a>0$
and $x-a<u<x$,
\begin{eqnarray*}
\{T_D(a)\in dt, T_U(a)>t, X_t=u\}=\{\tau_u\in dt, \sup_{s\le
t}X_s=u+a\},
\end{eqnarray*}
which suggests that for any $\lambda>0$,
\begin{eqnarray*}
&~&E_x\left[e^{-\lambda T_D(a)}\cdot
1_{\{T_D(a)<T_U(a),X_{T_D(a)}=u\}}\right]\\&=&\int_0^\infty
e^{-\lambda t}\cdot E_x\left[1_{\{T_D(a)\in dt, T_U(a)>t,
X_t=u\}}\right]\\&=&\int_0^\infty e^{-\lambda t}\cdot
E_x\left[1_{\{\tau_u\in dt,
\sup_{s\le t}X_s=u+a\}}\right]\\
&=&\int_0^\infty e^{-\lambda t}\cdot\frac{\partial}{\partial a}E_x\left[1_{\{\tau_u\in dt, \sup_{s\le t}X_s< u+a\}}\right]\\
&=&\int_0^\infty e^{-\lambda t}\cdot\frac{\partial}{\partial a}E_x\left[1_{\{\tau_u=dt, \tau_{u+a}>t\}}\right]\\
&=&\frac{\partial}{\partial
a}E_x\left[e^{-\lambda\tau_u}\cdot1_{\{\tau_{u}<\tau_{u+a}\}}\right].
\end{eqnarray*}
From Lemma 1 it follows that
\begin{eqnarray}\label{thm14}
E_x\left[e^{-\lambda T_D(a)}\cdot
1_{\{T_D(a)<T_U(a),X_{T_D(a)}=u\}}\right]&=&\frac{\partial}{\partial
a}\ell^X(u,u+a;x,\lambda).
\end{eqnarray}
Then integration of the above identity over the interval $(x-a,a)$
in $u$ yields (\ref{lequal}) and completes the proof of the
theorem.
\end{proof}

\subsection{The case of $a>b>0$}
We determine $L_x^X(\lambda; a,b)$ for $a>b>0$ in this paragraph. To prove the
main result we need following proposition.
\begin{proposition}
For $b>0$, $c<u$ satisfying $c,u+b\in I$, and $\lambda>0$, define
\begin{eqnarray}\label{hlap}
H_u^X(\lambda;b,c)&:=&E_u\left[e^{-\lambda\tau_{c}}\cdot
1_{\{\tau_c<T_U(b)\}}\right].
\end{eqnarray}
Then
\begin{eqnarray}
H_u^X(\lambda;b,c)&=&\exp\left[\int_c^u\left.\frac{\partial}{\partial
w}\right|_{w=v}\ell^X(v,v+b;w,\lambda)dv\right].
\end{eqnarray}
\end{proposition}
\begin{proof}
We follow the idea of \cite{HadjVallStam09}, \cite{JPL77}, , and partition the interval
$[c,u]$ into $k$ subintervals $\{[v_{i},v_{i-1}]; 1\le i\le k\}$ with
$u=v_0>v_1>\ldots>v_k=c$. Let $\Delta_k=\max_{1\le i\le
k}(v_{i-1}-v_{i})$ and assume $\Delta_k\to0$ as $k\to\infty$. As a
discrete approximation to $H_u^{X}(\lambda;b,c)$ defined by
(\ref{hlap}), compute
\begin{eqnarray*}
&~&E_u\left[e^{-\lambda\sum_{i=1}^k(\tau_{v_i}-\tau_{v_{i-1}})}\cdot1_{\left\{\textrm{after}~\tau_{v_{i-1}}, X_t~\textrm{hits}~ v_i~\textrm{before increasing to}~ v_i+b, 1\le i\le k\right\}}\right]\\
&=&\prod_{i=1}^k
E_{v_{i-1}}\left[e^{-\lambda\tau_{v_i}}\cdot1_{\{\tau_{v_i}<\tau_{v_i+b}\}}\right],
\end{eqnarray*}
where the last equality follows from the strong Markov property.

It will be shown that as $k\to\infty$ and $\Delta_k\to0$, the limit of above expression exists and does not depend on
the particular sequence of partition chosen and hence is equal to $H_u^X(\lambda;b,c)$. \\
\noindent By Lemma 1,
\begin{eqnarray*}
\prod_{i=1}^k
E_{v_{i-1}}\left[e^{-\lambda\tau_{v_i}}\cdot1_{\left\{\tau_{v_i}<\tau_{v_{i}+b}\right\}}\right]
&=&\prod_{i=1}^k \ell^X(v_i,v_i+b;v_{i-1},\lambda).
\end{eqnarray*}
Taking log gives us
\begin{eqnarray*}
&~&\sum_{i=1}^k\log\left[\ell^X(v_i,v_i+b;v_{i-1},\lambda)\right]\\
&=&\sum_{i=1}^k\log\left[\ell^X(v_i,v_i+b;v_i,\lambda)+\ell^X(v_i,v_i+b;v_{i-1},\lambda)-\ell^X(v_i,v_i+b;v_{i},\lambda)\right]\\
&=&\sum_{i=1}^k\log\left[1+\ell^X(v_i,v_i+b;v_{i-1},\lambda)-\ell^X(v_i,v_i+b;v_{i},\lambda)\right]\\
&=&\sum_{i=1}^k\left.\frac{\partial}{\partial w}\right|_{w=v_i}\ell^X(v_i,v_i+b;w,\lambda)\cdot(v_{i-1}-v_{i})+O(\Delta_k)\\
&\to&\int_c^u\left.\frac{\partial}{\partial
w}\right|_{w=v}\ell^X(v,v+b;w,\lambda)dv,\qquad\textrm{as $\Delta_k\to 0^+$},
\end{eqnarray*}
from which we obtain
\begin{eqnarray*}
H_u^{X}(\lambda;b,c)&=&\exp\left[\int_c^u\left.\frac{\partial}{\partial
w}\right|_{w=v}\ell^X(v,v+b;w,\lambda)dv\right].
\end{eqnarray*}
This completes the proof of Proposition 1.
\end{proof}
Now let us state and prove the main result in this
paragraph.
\begin{theorem}\label{thm2} For $a>b>0$ and $\lambda>0$, we have
\begin{eqnarray}\label{lapage}~~L_x^X(\lambda;a,b)&=&\int_{x-b}^{x}\frac{\partial}{\partial
b}\ell^X(u,u+b;x,\lambda)\cdot H_u^X(\lambda;b,u-a+b)du.
\end{eqnarray}
\end{theorem}
\begin{proof}
Any path in the event $\{T_D(a)<T_U(b)\}$ has the decomposition
\begin{enumerate}
\item $\{X_t; 0\le t\le T_D(b)\}$;
\item $\{X_{t+T_D(b)};0\le t\le T_D(a)-T_D(b)\}$.
\end{enumerate}
Conditioning on $X_{T_D(b)}=u$, the process in $2.$ starts at $u$, and decreases to $u-a+b$ before it incurs the drawup of $b$ units occurs. This gives rise to the representation
\begin{eqnarray}
T_D(a)=T_D(b)+\tau_{u-a+b}\circ\theta(T_D(b)).
\end{eqnarray}
Therefore, for $x-b<u<x$,
\begin{eqnarray}\label{311}
\nonumber&~&e^{-\lambda T_D(a)}\cdot 1_{\left\{T_D(a)<T_U(b),
X_{T_D(b)}=u\right\}}\\
\nonumber&=&e^{-\lambda
[T_D(b)+\tau_{u-a+b}\cdot\theta(T_D(b))]}\cdot
1_{\left\{T_D(a)<T_U(b), X_{T_D(b)}=u\right\}}\\
\nonumber&=&\underbrace{e^{-\lambda T_D(b)}\cdot 1_{\{T_D(b)<T_U(b),
X_{T_D(b)}=u\}}}_{\textrm{before}~T_D(b)}\\
&~&\times\underbrace{e^{-\lambda\tau_{u-a+b}\circ\theta(T_D(b))}\cdot1_{\left\{\tau_{u-a+b}\circ\theta(T_D(b))<T_U(b)\circ\theta(T_D(b))\right\}}}_{\textrm{after}~T_D(b)}.
\end{eqnarray}
To get the expectation of the above expression under $E_x[\cdot]$,
we first compute its conditional expectation under
$E_u[\cdot|\mathcal{F}_{T_D(b)}]$. By the strong Markov property, the
factor ``before $T_D(b)$'' is deterministic under
$E_u[\cdot|\mathcal{F}_{T_D(b)}]$, and the factor ``after $T_D(b)$''
has conditional expectation
\begin{eqnarray*}
&~&E_u\left[e^{-\lambda\tau_{u-a+b}\circ\theta(T_D(b))}\cdot1_{\{\tau_{u-a+b}\circ\theta(T_D(b))<T_U(b)\circ\theta(T_D(b))\}}\left.\right|\mathcal{F}_{T_D(b)}\right]\\
&=&E_u\left[e^{-\lambda\tau_{u-a+b}}\cdot1_{\{\tau_{u-a+b}<T_U(b)\}}\right],
\end{eqnarray*}
which, by Proposition 1, is equal to $H_u^X(\lambda; b,u-a+b)$.
Taking the expectation of (\ref{311}) under $E_x[\cdot]$, and using (\ref{thm14}), we obtain
\begin{eqnarray}\label{thm22}
\nonumber&~&E_x\left[e^{-\lambda
T_D(a)}\cdot1_{\left\{T_D(a)<T_U(b),
X_{T_D(b)}=u\right\}}\right]\\
\nonumber&=&E_x\left[e^{-\lambda T_D(b)}\cdot 1_{\{T_D(b)<T_U(b),
X_{T_D(b)}=u\}}\right]\cdot H_u^X(\lambda;b,u-a+b)\\
&=&\frac{\partial}{\partial b}\ell^X(u,u+b;x,\lambda)\cdot
H_u^X(\lambda;b,u-a+b).
\end{eqnarray}
The integration of the above identity over
the interval $(x-b,x)$ in $u$ yields (\ref{lapage}) and completes the
proof of Theorem \ref{thm2}.
\end{proof}

\subsection{The case of $b>a>0$}
We determine $L_x^X(\lambda;a,b)$ for $b>a>0$ in this paragraph. To prove the
main result we need the following proposition.
\begin{proposition}
For any $a>0$ and $x\in I$ satisfying $x-a\in I$, and $\lambda>0$, define
\begin{eqnarray}
J_x^X(\lambda;a):=E_x\left[e^{-\lambda T_D(a)}\right].
\end{eqnarray}
Then
\begin{align*}
 &J_x^X(\lambda;a)\\
=&-\int_x^\infty 1_{I}(u)\cdot\left.\frac{\partial}{\partial w}\right|_{w=u}\ell^{X}(u-a,u;w,\lambda)\cdot e^{-\int_x^u\left.\frac{\partial}{\partial w}\right|_{w=v}\ell^X(v,v-a;w,\lambda)dv}du.
\end{align*}
\end{proposition}
\begin{proof}
See \cite{JPL77}, page 602.
\end{proof}

Now let us state and prove the main result in this paragraph.
\begin{theorem}\label{thm3}
For $b>a>0$ and $\lambda>0$, we have
\begin{eqnarray}\label{thm3f}
\nonumber ~~L_x^X(\lambda;a,b)&=&J_x^{X}(\lambda;a)-\int_{x}^{x+a}dv\frac{\partial}{\partial
a}\ell^{2x-X}(2x-v,2x-v+a;x,\lambda)\\
&~&\times H_{2x-v}^{2x-X}(\lambda;a,2x-v-b+a)\cdot
J_{v+b-a}^X(\lambda;a).
\end{eqnarray}
\end{theorem}
\begin{proof}
First, it is easily seen that for $b\ge a>0$,
\begin{eqnarray}
\nonumber L_x^X(\lambda;a,b)&=&J_x^X(\lambda;a)-E_x\left[e^{-\lambda
T_D(a)}\cdot1_{\left\{\sup_{s\le T_D(a)}DU_s\ge b\right\}}\right].
\end{eqnarray}
Therefore, to prove (\ref{thm3f}), it suffices to show that
\begin{eqnarray}\label{largedu}
\nonumber E_x\left[e^{-\lambda
T_D(a)}\cdot1_{\left\{\sup_{s\le T_D(a)}DU_s\ge b\right\}}\right]
=\int_{x}^{x+a}dv\frac{\partial}{\partial
a}\ell^{2x-X}(2x-v,2x-v+a;x,\lambda)\\
\times H_{2x-v}^{2x-X}(\lambda;a,2x-v-b+a)\cdot J_{v+b-a}^X(\lambda;a).~
\end{eqnarray}
To do so we observe that for $b\ge a$,
\begin{eqnarray}\label{27}
~~~~\qquad~~ E_x\left[e^{-\lambda T_D(a)}\cdot1_{\{\sup_{t\le T_D(a)}DU_t\ge b\}}\right]=E_x\left[e^{-\lambda T_D(a)}\cdot1_{\{T_U(b)<T_D(a)\}}\right].
\end{eqnarray}
And we are going to calculate the right hand side of (\ref{27}) in the sequel.

Consider the path decomposition for any path in the event $\{T_U(b)<T_D(a)\}$. We have
\begin{enumerate}
\item $\{X_t; 0\le t\le T_U(b)\}$;
\item $\{X_{t+T_U(b)}; 0\le t\le T_D(a)-T_U(b)\}$.
\end{enumerate}
Intuitively, before time $T_U(b)$, the process experiences no drawdown of $a$
units and the first drawup of $b$ units occurs at
$T_U(b)$, when the process also reaches a new maximum; thereafter, the process has a drawdown of $a$ units at time $T_D(a)$. Thus for any path in the event $\{T_U(b)<T_D(a)\}$ we have
\begin{eqnarray} T_D(a)=T_U(b)+T_D(a)\circ\theta(T_U(b)).
\end{eqnarray}
Therefore, for $b\ge a$ and $x<v<x+a$,
\begin{eqnarray}\label{thm31}
\qquad~~~~ &~&E_x\left[e^{-\lambda T_D(a)}\cdot 1_{\{
\sup_{t\le T_D(a)}DU_t\ge b, X_{T_U(a)}=v\}}\right]\\
\nonumber&=&E_x\left[e^{-\lambda
T_U(b)+T_D(a)\circ\theta(T_U(b))}\cdot 1_{\{ \sup_{t\le
T_D(a)} DU_t\ge b,
X_{T_U(a)}=v\}}\right]\\
\nonumber&=&E_x\left[\underbrace{e^{-\lambda T_U(b)}\cdot
1_{\{T_U(b)<T_D(a), X_{T_U(a)}=v\}}}_{\textrm{before}~T_U(b)}\times
\underbrace{e^{-\lambda T_D(a)\circ\theta(T_U(b))}}_{\textrm{after}~T_U(b)}\right]\\
\nonumber&=&E_x\left[e^{-\lambda T_U(b)}\cdot 1_{\{T_U(b)<T_D(a),
X_{T_U(a)}=v\}}E_{v+b-a}\left[e^{-\lambda T_D(a)}\right]\right]\\
\nonumber&=&E_x\left[e^{-\lambda T_U(b)}\cdot 1_{\{T_U(b)<T_D(a),
X_{T_U(a)}=v\}}\cdot J_{v+b-a}^X(\lambda;a)\right]\\
\nonumber&=&E_x\left[e^{-\lambda T_U(b)}\cdot 1_{\{T_U(b)<T_D(a),
X_{T_U(a)}=v\}}\right]\cdot J_{v+b-a}^X(\lambda;a),
\end{eqnarray}
where the third equality follows from the strong Markov property.
The expectation in the last line can be determined as follows.
Note that for the process $\{Y_t=2x-X_t, t\ge 0\}$,
\begin{eqnarray*}
dY_t=-\mu(2x-Y_t)dt+\sigma(2x-Y_t)dB_t^{'},~Y_0=x
\end{eqnarray*}
with $B_t^{'}=-B_t$, the vector of random variables $\left(T_D^Y(b), T_U^Y(a), 2x-Y_{T_D^Y(a)}\right)$\footnote{$T_D^Y(b)$ and $T_U^Y(a)$ are the drawdown and drawup of the process $\{Y_t; t\ge 0\}$ respectively.}
has the same law as the vector of random variables $\left(T_D(a), T_U(b), X_{T_U(a)}\right)$ for
$\{X_t; t\ge 0\}$ under $E_x[\cdot]$. So we know from (\ref{thm22})
that
\begin{eqnarray}\label{thm32}
\quad&~&E_x\left[e^{-\lambda T_U(b)}\cdot 1_{\{T_U(b)<T_D(a),
X_{T_U(a)}=v\}}\right]~\textrm{for}~X_t\\
\nonumber &=&E_x\left[e^{-\lambda T_D^Y(b)}\cdot 1_{\{T_D^Y(b)<T_U^Y(a),
Y_{T_D^Y(a)}=2x-v\}}\right]~\textrm{for}~ Y_t\\
\nonumber&=&\frac{\partial}{\partial
a}\ell^{2x-X}(2x-v,2x-v+a;x,\lambda)\cdot
H_{2x-v}^{2x-X}(\lambda;a,2x-v-b+a).
\end{eqnarray}
The integration of (\ref{thm31}) over the interval
$(x,x+a)$ in $v$ yields (\ref{largedu}) and completes the proof of Theorem \ref{thm3}.
\end{proof}

We now proceed to treat separately the case of a drifted Brownian motion model.

\section{A study of Brownian motion}
In this section, we apply the results in Theorem \ref{thm1}, Theorem \ref{thm2} and
Theorem \ref{thm3} to a drifted Brownian motion model and calculate the probability density of the drawdown of $a$ units when it precedes the drawup of $b$ units.

First it is easily seen that $I=(-\infty,\infty)$  in a drifted Brownian motion model. The function $\ell^{X}(y,z;x,\lambda)$ for $X_t=x+\mu t+\sigma W_t$ can be found in \cite{BoroSalm02} page 295:
\begin{eqnarray}
\ell^X(y,z;x,\lambda)&=&\frac{\sinh[(z-x)S_{\mu,\sigma}(\lambda)]}{\sinh[(z-y)S_{\mu,\sigma}(\lambda)]}e^{\frac{\mu(y-x)}{\sigma^2}},
\end{eqnarray}
where $S_{\mu,\sigma}(\lambda)=\sqrt{(2\lambda/\sigma^2)+(\mu^2/\sigma^4)}$. Thus the Laplace transforms in Theorem \ref{thm1}, Theorem \ref{thm2} and Theorem \ref{thm3} can be calculated explicitly as:\\
1. $a=b>0$:
\begin{eqnarray}\label{a=b}
\nonumber L_0^X(\lambda;a)&=&\frac{S_{\mu,\sigma}(\lambda)}{(2\lambda/\sigma^2)}\left\{\frac{e^{-\frac{\mu a}{\sigma^2}}\left[S_{\mu,\sigma}(\lambda)\coth[aS_{\mu,\sigma}(\lambda)]+\frac{\mu}{\sigma}\right]}{\sinh[a S_{\mu,\sigma}(\lambda)]}\right.\\
&~&\left.-\frac{S_{\mu,\sigma}(\lambda)}{\sinh^2[aS_{\mu,\sigma}(\lambda)]}\right\};
\end{eqnarray}
\noindent2. $a>b>0$:
\begin{eqnarray}\label{a>b}
L_0^X(\lambda;a,b)=L_0^X(\lambda; b)\cdot\exp\left[T_{\mu,\sigma}(\lambda;b)(a-b)\right],
\end{eqnarray}
where
\begin{eqnarray}\label{exp}
T_{\mu,\sigma}(\lambda;b)=-\frac{\mu}{\sigma^2}-S_{\mu,\sigma}(\lambda)\coth[bS_{\mu,\sigma}(\lambda)];
\end{eqnarray}
\noindent3. $b>a>0$:
\begin{eqnarray}\label{b>a}
L_0^X(\lambda;a,b)=\left[1-L_0^{-X}(\lambda;a)\cdot e^{T_{-\mu,\sigma}(\lambda;a)(b-a)}\right]\cdot J_0^X(\lambda;a),
\end{eqnarray}
where
\begin{eqnarray}
\nonumber J_0^X(\lambda;a)&=&\frac{S_{\mu,\sigma}(\lambda) e^{-\frac{\mu a}{\sigma^2}}}{S_{\mu,\sigma}(\lambda)\cosh[aS_{\mu,\sigma}(\lambda)]-(\mu/\sigma^2)\sinh[a S_{\mu,\sigma}(\lambda)]}\\
&=&-\frac{S_{\mu,\sigma}(\lambda;a)e^{-\frac{\mu  a}{\sigma^2}}}{\sinh[aS_{\mu,\sigma}(\lambda;a)]}\cdot\frac{1}{T_{-\mu,\sigma}(\lambda;a)}.
\end{eqnarray}

One can easily obtain several known results from (\ref{a=b}), (\ref{a>b}) and (\ref{b>a}). First, by letting $\lambda\to0^+$, the formulas coincide with the probability results in Hadjiliadis \& Vecer \cite{HadjVece06}. Second, by letting $b\to\infty$ in (\ref{b>a}), one obtains the Laplace transform of $T_D(a)$, $J_0^X(\lambda;a)$.

Moreover, we can invert (\ref{a>b}) analytically to obtain the density $P(T_D(a)\in dt, T_U(b)>t)$ for any $a\ge b>0$.  In fact we have
\begin{theorem}\label{thm4}
Define $p^{(\mu)}(t;a,b)dt=P(T_D(a)\in dt, T_U(b)>t)$ for $a\ge b>0$, then
\begin{eqnarray}\label{thm4f}
\nonumber p^{(\mu)}(t;a,b)&=&\frac{2}{t}e^{-\frac{\mu^2t}{2\sigma^2}-\frac{\mu(a-b)}{\sigma^2}}\sum_{m,n=0}^{\infty}\frac{(m+n+1)!}{(m+1)!m!n!}\left[\frac{2(a-b)}{\sigma\sqrt{t}}\right]^m\\
\nonumber&~&\times\left[2F_{m,n}^{(1)}(t)-e^{-\frac{\mu b}{\sigma^2}}F_{m,n}^{(2)}(t)-\delta_{n}e^{-\frac{\mu b}{\sigma^2}}F_m^{(3)}(t)\right]\\
\nonumber&~&+\frac{2\mu^2}{\sigma^2}e^{-\frac{\mu(a-b)}{\sigma^2}}\sum_{m,n=0}^{\infty}\frac{(m+n+1)!}{(m+1)!m!n!}\left[\frac{2\mu(a-b)}{\sigma^2}\right]^m\times\\
\nonumber&~&\left[G_{m+n+\frac{1}{2}}^{(\mu)}(t)+(-1)^mG_{m+n+\frac{1}{2}}^{(-\mu)}(t)-e^{-\frac{\mu b}{\sigma^2}}G_{m+n+1}^{(\mu)}(t)\right.\\
&~&\left.-(-1)^me^{-\frac{\mu b}{\sigma^2}}G_{m+n}^{(-\mu)}(t)\right],\quad~
\end{eqnarray}
where $\delta_{n}$ is the Kronecker delta and
\begin{align*}
F_{m,n}^{(1)}(t)=&\sum_{k=0}^{\left\lfloor \frac{m+1}{2}\right\rfloor}\left[\frac{\mu\sqrt{t}}{\sigma}\right]^{2k}\phi^{(m+1-2k)}\left[\frac{(2m+2n+1)b+a}{\sigma\sqrt{t}}\right],\\
F_{m,n}^{(2)}(t)=&\sum_{k=0}^{m+1}\left[\frac{\mu\sqrt{t}}{\sigma}\right]^k\left[1+(-1)^k\frac{m+n+2}{n+1}\right]\phi^{(m+1-k)}\left[\frac{2(m+n+1)b+a}{\sigma\sqrt{t}}\right]\\
F_m^{(3)}(t)=&\sum_{k=0}^{m+1}\left[-\frac{\mu\sqrt{t}}{\sigma}\right]^k\phi^{(m+1-k)}\left[\frac{2mb+a}{\sigma\sqrt{t}}\right],\\
G_m^{(\mu)}(t)=&e^{\frac{\mu(2mb+a)}{\sigma^2}}\Phi\left[\frac{2mb+a+\mu t}{\sigma\sqrt{t}}\right],
\end{align*}
with $\phi$ and $\Phi$ being the standard normal probability density and cumulative distribution respectively. $\phi^{(k)}$ is the $k$-th derivative of $\phi$.
\end{theorem}
\begin{proof}
We start by rewriting (\ref{a>b}) in a more tractable way,
\begin{eqnarray}\label{integral}
\int_{-b}^0due^{\frac{\mu u}{\sigma^2}}\frac{S_{\mu,\sigma}(\lambda)\sinh[(-u)S_{\mu,\sigma}(\lambda)]}{\sinh^2[bS_{\mu,\sigma}(\lambda)]}\cdot\exp\left[T_{\mu,\sigma}(\lambda;b)(a-b)\right].\quad
\end{eqnarray}
By using the first formula on page 643 of \cite{BoroSalm02}, in their notation, we obtain the inverse Laplace transform of the integrand in (\ref{integral})
\begin{eqnarray*}
\frac{\sigma^2}{2}e^{-\frac{\mu^2t}{2\sigma^2}+\frac{\mu(u+b-a)}{\sigma^2}}\left[es_{\sigma^2t}(1,2,b,u,a-b)-es_{\sigma^2t}(1,2,b,-u,a-b)\right].
\end{eqnarray*}
After some simple manipulation, the above expression becomes
\begin{eqnarray}\label{pdfu}
\frac{2e^{-\frac{\mu^2t}{2\sigma^2}+\frac{\mu(u+b-a)}{\sigma^2}}}{\sigma \sqrt{t^3}}\sum_{m,n=0}^{\infty}\frac{(m+n+1)!}{(m+1)!m!n!}\left[\frac{2(a-b)}{\sigma\sqrt{t}}\right]^m\times\\
\nonumber\left\{\phi^{(m+2)}\left[\frac{(2m+2n+1)b+a+u}{\sigma\sqrt{t}}\right]
-\phi^{(m+2)}\left[\frac{(2m+2n+1)b+a-u}{\sigma\sqrt{t}}\right]\right\}.\quad~
\end{eqnarray}
Formula (\ref{thm4f}) follows from integration of (\ref{pdfu}) over $(-b,0)$ in $u$.
\end{proof}

One can let $a=b$ in (\ref{pdfu}) to get a similar joint probability density as that in Proposition 1 of \cite{ZhanHadj09}.

Similarly, we can invert (\ref{b>a}) analytically to get the joint density
\begin{eqnarray*}P(T_D(a)\in dt, \sup_{s\le t}DU_s\in a+dz)&=&\frac{\partial}{\partial z}p^{(\mu)}(t;a,a+z)dtdz\qquad\forall a,z>0.\end{eqnarray*}
In particular we have
\begin{theorem}\label{thm5}
For a Brownian motion with constant drift $\mu$ and constant volatility $\sigma$, any $a,z>0$, we have
\begin{eqnarray}\label{thm5f}
\frac{\partial}{\partial z}\nonumber p^{(\mu)}(t;a,a+z)&=&-\frac{4e^{-\frac{\mu^2t}{2\sigma^2}+
\frac{\mu(z-a)}{\sigma^2}}}{\sigma\sqrt{t^3}}\sum_{m,n=0}^{\infty}\frac{(m+n+2)!}{(m+2)!m!n!}\left(\frac{2z}{\sigma\sqrt{t}}\right)^m\\
\nonumber&~&\times\left[2F_{m,n}^{(1)}(t,z)-e^{\frac{\mu a}{\sigma^2}}F_{m,n}^{(2)}(t,z)-\delta_ne^{\frac{\mu a}{\sigma^2}}F_m^{(3)}(t,z)\right]\\
\nonumber&~&-\frac{4\mu^3}{\sigma^4}e^{\frac{\mu(z-a)}{\sigma^2}}\sum_{m,n=0}^{\infty}\frac{(m+n+2)!}{(m+2)!m!n!}\left(\frac{2\mu z}{\sigma^2}\right)^m\times\\
\nonumber&~&\left[G_{m+n+\frac{1}{2}}^{(\mu)}(t,z)-(-1)^m G_{m+n+\frac{1}{2}}^{(-\mu)}(t,z)-e^{\frac{\mu a}{\sigma^2}}G_{m+n}^{(\mu)}(t,z)\right.\\
&~&\left.+(-1)^me^{\frac{\mu a}{\sigma^2}}G_{m+n+1}^{(-\mu)}(t,z)\right]
\end{eqnarray}
where
\begin{align*}
F_{m,n}^{(1)}(t,z)=&\sum_{k=0}^{\left\lfloor \frac{m+2}{2}\right\rfloor}\left[\frac{\mu\sqrt{t}}{\sigma}\right]^{2k}\phi^{(m+2-2k)}\left[\frac{(2m+2n+3)a+z}{\sigma\sqrt{t}}\right],\\
F_{m,n}^{(2)}(t,z)=&\sum_{k=0}^{m+2}\left[\frac{\mu\sqrt{t}}{\sigma}\right]^{k}\left[(-1)^k+\frac{m+n+3}{n+1}\right]\phi^{(m+2-k)}\left[\frac{(2m+2n+4)a+z}{\sigma\sqrt{t}}\right]\qquad\\
F_m^{(3)}(t,z)=&\sum_{k=0}^{m+2}\left[\frac{\mu\sqrt{t}}{\sigma}\right]^k\phi^{(m+2-k)}\left[\frac{(2m+2)a+z}{\sigma\sqrt{t}}\right],\\
G_m^{(\mu)}(t,z)=&e^{\frac{\mu[2(m+1)a+z]}{\sigma^2}}\Phi\left[\frac{2(m+1)a+z+\mu t}{\sigma\sqrt{t}}\right].
\end{align*}\end{theorem}
\begin{proof}
We start from the equality
\begin{align*}
&L_0^X(\lambda;a,b)\\
=&J_0^X(\lambda;a)-L_0^{-X}(\lambda;a)e^{T_{-\mu,\sigma}(\lambda;a)(b-a)}J_0^X(\lambda;a)\\
=&J_0^X(\lambda;a)-L_0^{-X}(\lambda;a)J_0^X(\lambda;a)+\frac{S_{\mu,\sigma}(\lambda;a)L_0^{-X}(\lambda;a)}{\sinh[aS_{\mu,\sigma}(\lambda;a)]e^{\frac{\mu a}{\sigma^2}}}\int_0^{b-a}e^{T_{-\mu,\sigma}(\lambda;a)z}dz\\
=&L_0^{X}(\lambda;a)+\int_{-a}^0du\int_0^{b-a}dz\frac{[S_{\mu,\sigma}(\lambda;a)]^2\sinh[(-u)S_{\mu,\sigma}(\lambda;a)]}{\sinh^3[aS_{\mu,\sigma}(\lambda;a)]e^{\frac{\mu(u+a)}{\sigma^2}}}e^{T_{-\mu,\sigma}(\lambda;a)z}\qquad
\end{align*}
By using the first formula on page 643 of \cite{BoroSalm02}, the integrand in the last line has inverse Laplace transform
\begin{eqnarray*}
\frac{\sigma^2}{2}e^{-\frac{\mu^2t}{2\sigma^2}-\frac{\mu(u-z+a)}{\sigma^2}}[es_{\sigma^2t}(2,3,a,u,z)-es_{\sigma^2t}(2,3,a,-u,z)].
\end{eqnarray*}
After some simple manipulation, the above expression becomes
\begin{eqnarray}\label{54}
\frac{4e^{-\frac{\mu^2t}{2\sigma^2}-\frac{\mu(u-z+a)}{\sigma^2}}}{\sigma^2t^2}\sum_{m,n=0}^{\infty}\frac{(m+n+2)!}{(m+2)!m!n!}\left(\frac{2z}{\sigma\sqrt{t}}\right)^m\times\\
\nonumber\left\{\phi^{(m+3)}\left[\frac{(2m+2n+3)a+z-u}{\sigma\sqrt{t}}\right]-\phi^{(m+3)}\left[\frac{(2m+2n+3)a+z+u}{\sigma\sqrt{t}}\right]\right\}.
\end{eqnarray}
The integration of (\ref{54}) over $(-a,0)$ in $u$ yields (\ref{thm5f}) and completes the proof.
\end{proof}

\section{Application}
\subsection{Relative drawdowns and relative drawups of stock prices}
Consider the case of a stock with geometric Brownian motion dynamics under a probability measure $P$:
\begin{eqnarray}\label{gbm}
dS_t & = & \mu S_t dt+ \sigma S_t dW_t, ~ S_0=1.
\end{eqnarray}

Using Theorem \ref{thm4} and Theorem \ref{thm5}, we are in the position to address the following question:

{\bf What is the probability that this stock would drop by $\bm(\bm100\bm\times\bm\alpha\bm)\bm\%$ from its historical high before it incurs a rise of $\bm(\bm100\bm\times\bm\beta)\bm\%$ from its historical low in a pre-specified plan horizon $T$?}

First observe that
\begin{eqnarray}
\label{logdynamics}
d\log S_t & = & \nu dt + \sigma dW_t, ~\log S_0=0,
\end{eqnarray}
where $\nu=\mu-\frac{1}{2} \sigma^2$ represents the logarithm of the return of the stock.

Now define the running maximum and the running minimum of the stock process $\{S_t\}$
\begin{eqnarray*}
M_t =  \sup_{s \le t} S_s,\qquad
N_t =  \inf_{s \le t} S_s.
\end{eqnarray*}
We also let $U_D(\alpha)$ be the first time the stock drops by ($100 \times\alpha$)\%
from its historical high and $U_R(\beta)$ the first time that the stock rises by an amount equal to ($100 \times \beta$)\% from its historical low. That is,
\begin{eqnarray}
\label{U1}
U_D(\alpha) & = & \inf\{\left.t \ge 0\right| S_t= (1-\alpha) M_t\}, \\
\label{U2}
U_R(\beta) & = & \inf\{\left.t \ge 0\right| S_t= (1+\beta) N_t\}.
\end{eqnarray}

Thus, it is possible to calculate the exact expression for the probability that a percentage relative drop of ($100 \times\alpha$)\% precedes a relative rise of ($100 \times\beta$)\% by noticing that
\begin{eqnarray}
\label{g2a}
\quad\left\{\begin{array}{ll}U_D(\alpha)=&T_D(-\log(1-\alpha))\\ U_R(\beta)=&T_U(\log(1+\beta))\end{array}.\right.
\end{eqnarray}
And this probability can be calculated explicitly as
\begin{eqnarray*}
P(U_D(\alpha)\wedge T<U_R(\beta)\wedge T)&=&\int_0^Tp^{(\nu)}(t;-\log(1-\alpha),\log(1+\beta))dt.
\end{eqnarray*}

Moreover, a digital option on the event that the relative drawdown precedes the relative drawup can also be perceived as a means of protection against adverse movements in the market. In particular, the discounted payoff of this digital option can be written as
\begin{eqnarray}
PO(\alpha,\beta)&=&e^{-rt}\cdot1_{\{U_D(\alpha)\in dt, U_R(\beta)>t\}}\cdot1_{\{t\le T\}},
\end{eqnarray}
where $r>0$ is the risk-free interest rate and $T$ is the maturity of the option.

Under a risk-neutral measure $Q$, the stock price and its logarithm have the following dynamics respectively,
\begin{eqnarray}
dS_t&=&rS_tdt+\sigma S_tdW_t, ~S_0=1,\\
d\log S_t&=&\nu^{'}dt+\sigma dW_t, ~\log S_0=0,
\end{eqnarray}
where $\nu^{'}=r-\frac{1}{2}\sigma^2$.

Using (\ref{g2a}) and our results we are able to derive the risk-neutral price at time 0 of this digital option:

In the case of a perpetual option (see \cite{KaraShre01}), the risk-neutral price of the digital option is already given by the Laplace transform (\ref{a=b}), (\ref{a>b}) and (\ref{b>a}). In particular,
\begin{eqnarray*}Q[PO(\alpha,\beta)]&=&L_0^{\log S}(r;-\log(1-\alpha),\log(1+\beta)).\end{eqnarray*}
In the case of a finite life option maturing at time $T<\infty$, we can apply the densities (\ref{thm4f}) and (\ref{thm5f}) to calculate the risk-neutral price.
\begin{enumerate}
\item $(1-\alpha)(1+\beta)\le 1$:
\begin{eqnarray}\label{price}
\qquad Q[PO(\alpha,\beta)]&=&\int_0^T e^{-rt}p^{(\nu^{'})}(t;-\log(1-\alpha),\log(1+\beta))dt;
\end{eqnarray}
\item $\delta=(1-\alpha)(1+\beta)>1$:
\begin{eqnarray}\label{price1}
&~&Q[PO(\alpha,\beta)]-Q[PO(\alpha,\alpha/(1-\alpha))]\\
\nonumber&=&\int_0^T e^{-rt}\left[\int_0^{\log\delta}\frac{\partial}{\partial z}p^{(\nu^{'})}(t;-\log(1-\alpha),-\log(1-\alpha)+z)dz\right]dt.
\end{eqnarray}
\end{enumerate}

\subsection{The problem of transient signal detection and identification of two sided changes}
In this example, we present the problem of transient signal detection and identification of two-sided changes in the drift of a general diffusion process. More specifically, we give the formulas for the probabilities of misidentifying the direction of the signal both in the case of exponential life transient signals and in the case of deterministic life transient signals. In particular, let $\{X_t, t\ge 0\}$ be a diffusion process with the initial value $X_0=x$ and the following dynamics up to a deterministic time $\tau$:
\begin{eqnarray}\label{nominal}
dX_t&=&\sigma(X_t)dW_t,\quad t\le\tau.
\end{eqnarray}
For $t>\tau$, the process evolves according to one of the following stochastic differential equations:
\begin{eqnarray}
\label{pchange}
dX_t&=&\mu(X_t)dt+\sigma(X_t)dW_t\quad t>\tau,\\
\label{nchange}dX_t&=&-\mu(X_t)dt+\sigma(X_t)dW_t\quad t>\tau,
\end{eqnarray}
with initial condition $y=X_\tau$.

The time of the regime change, $\tau$, is deterministic but unknown. We observe the process $\{X_t, t\ge0\}$ sequentially and our goal is to detect the time of onset of the signal, as well as possibly identity its direction, before the signal disappears.

Let us denote by $P_x^{\tau, +}$ and $P_x^{\tau,-}$ the probability measures generated on the space of continuous functions $C[0,\infty)$ by the process $\{X_t, t\ge 0\}$, if the regime changes at time $\tau$ from (\ref{nominal}) to (\ref{pchange}) and from (\ref{nominal}) to (\ref{nchange}), respectively.

In this context suppose that the drawdown of $a$ units, $T_D(a)$, can be used as a means of detecting the change in the dynamics of $\{X_t, t\ge 0\}$ from (\ref{nominal}) to (\ref{nchange}). Then the drawup of $b$ units, $T_U(b)$ may be used as a means of detecting the change in the dynamics of $\{X_t, t\ge 0\}$ from (\ref{nominal}) to (\ref{pchange}). For example, in the case that $\mu(\cdot)=\mu>0$, it is easy to see that the drawup could be used as a means of detecting a change from (\ref{nominal}) to (\ref{pchange}), while the drawdown would provide a means of detecting a change from (\ref{nominal}) to (\ref{nchange}). In particular, $T(a,b)=\min\{T_D(a), T_U(b)\}$, also known as the two-sided CUSUM (see Khan \cite{Khan08}), has been extensively used as a means of detecting two-sided alternatives in the drift (see Barnard \cite{Barn59}, Dobben \cite{Dobben68}, Bissell \cite{Biss69}, Woodall \cite{Wood84}, and Poor \& Hadjiliadis \cite{PHadj08}). In the special case in which $\mu(\cdot)=\mu>0$ and there is no reason to believe that a change from (\ref{nominal}) to (\ref{nchange}) is more or less likely than a change from (\ref{nominal}) to (\ref{pchange}), it is natural to use thresholds $a=b$. However, in the general case different thresholds $a$ and $b$ could be used.

In many applications in engineering, the life of the signal after its onset is often limited. The lifetime of the signal is also random and may not depend on the dynamics of the underlying process. In the case of exponential life transient signals, the signals are present (after $\tau$) for a period of time $\zeta$, where $\zeta$ is an independent exponentially distributed random variable with parameter $\lambda>0$. Theorems \ref{thm1}, \ref{thm2} and \ref{thm3} can be used to compute the probability of sequential misidentification of the signal in the case that the onset of the signal occurs at time 0. More specifically,
\begin{eqnarray}\label{fap}
\nonumber P_x^{0,+}(T_D(a)\wedge\zeta<T_U(b)\wedge\zeta)&=&\int_0^\infty P_x^{0,+}(T_D(a)\wedge t<T_U(b)\wedge t)\cdot \lambda e^{-\lambda t}dt\\
\nonumber&=&\int_0^\infty e^{-\lambda t}P_x^{0,+}(T_D(a)\in dt, T_U(b)>t)dt\\
&=&L_x^{X^{0,+}}(\lambda; a, b),
\end{eqnarray}
where $X^{0,+}$ follows (\ref{pchange}) with $\tau=0$, expresses the probability that an alarm indicating that the regime switched to (\ref{nchange}) will occur before $\zeta$ and an alarm indicating that the regime switched to (\ref{nchange}) while in fact (\ref{pchange}) is the true regime. Thus (\ref{fap}) can be seen as the probability of a misidentification. Moreover, in the case that the density of the random variable $X_\tau$ admits a closed-form representation, we can also compute
\begin{eqnarray}
&~&\int P_y^{\tau,+}(T_D(a)\circ\theta(\tau)\wedge\zeta<T_U(b)\circ\theta(\tau)\wedge\zeta)f_{X_\tau}(y|x)dy\\
\nonumber&=&\int L_y^{X^{0,+}}(\lambda, a, b)f_{X_\tau}(y|x)dy,
\end{eqnarray}
which can be interpreted as the aggregate probability (or unconditional probability) of a misidentification for any given change-point $\tau$.

In the case of deterministic life transient signals, the signals are present (after $\tau$) for a finite period of time $T$. Using Theorem \ref{thm4} we are still able to compute the probability of misidentification for Brownian motion ($\sigma(\cdot)=\sigma>0, \mu(\cdot)=\mu$). More specifically,
\begin{eqnarray}
\qquad P_x^{\tau,+}(T_D(a)\circ\theta(\tau)\wedge T<T_U(a)\circ\theta(\tau)\wedge T)=\int_0^Tp^{(\mu)}(t; a, a)dt,
\end{eqnarray}
expresses the probability of misidentification for any given change-point $\tau$.

\section{Conclusion}
In this paper we derive a closed-form expression for the Laplace transform of the drawdown of $a$ units when it precedes the drawup of $b$ units for a general diffusion process. We then derive the probability density of a drawdown when it precedes a drawup in the special case of a drifted Brownian motion model by inverting the Laplace transform. Although several authors in the literature have studied drawdowns and drawups (\cite{Taylor}, \cite{JPL77}, \cite{HadjVece06}, \cite{Posp09}, \cite{SalmVall}, \cite{ZhanHadj09}), this paper summarizes the probabilistic properties of a drawdown on the event that it precedes a drawup for a general diffusion process. These results are of practical interest in two main areas: financial risk-management and transient signal detection and identification.

\end{document}